\documentclass{article}
\usepackage[margin=1.35in]{geometry} 
\usepackage{graphicx} 

\pdfoutput = 1

\usepackage{amsmath,amsthm,amssymb,amsfonts, fancyhdr, comment, parskip, tikz, caption, subcaption, graphicx, xcolor}
\usepackage[colorlinks,hypertexnames=false]{hyperref}
\usetikzlibrary{graphs}
\usetikzlibrary{shadings}
\usetikzlibrary{shapes.geometric}
\usetikzlibrary{arrows}
\usepackage{mathtools}
\usepackage{standalone}
\usepackage{upgreek}
\usepackage{diagbox}
\usepackage{booktabs} 
\usepackage{bm}
\usepackage{tikz-cd}
\usepackage{enumerate}
\usepackage{enumitem}
\usepackage{soul}
\usepackage{appendix}

\usepackage{xcolor}

\usepackage{listings}

\definecolor{codegreen}{rgb}{0,0.45,0}
\definecolor{codegray}{rgb}{0.45,0.45,0.45}
\definecolor{codepurple}{rgb}{0.55,0,0.55}
\definecolor{backcolour}{rgb}{0.97,0.97,0.97}

\usetikzlibrary{arrows,decorations.pathmorphing,backgrounds,positioning,fit}
\tikzstyle{dot} = [fill=black!25,inner sep=0.5mm,circle,draw,minimum size=1mm]

\usepackage{thm-restate}
\usepackage{thmtools}
\newtheorem{theorem}{Theorem}[section]
\newtheorem{proposition}[theorem]{Proposition}
\newtheorem{lemma}[theorem]{Lemma}

\newtheorem{corollary}[theorem]{Corollary}
\newtheorem{conjecture}[theorem]{Conjecture}

\newcommand{\R}{\mathbb{R}}

\newcommand{\cF}{\mathcal{F}}

\newcommand{\cC}{\mathcal{C}}

\newcommand{\cH}{\mathcal{H}}

\DeclareMathOperator{\KG}{KG}
\newcommand{\bstable}[3]{\binom{[#1]}{#2}_{#3\text{-stable}}}
\newcommand{\kgstable}[3]{\KG(#1, #2)_{#3\text{-stable}}}

\newcommand{\KGk}{\KG(n,k)_{3\text{-stable}}}

\newcommand{\sfK}{\textsf{K}}
\newcommand{\bS}{\mathbb{S}}

\DeclareMathOperator{\conv}{conv}

\title{The chromatic number of 3-stable Kneser graphs}
\date{}
\author{Wei-Chia Chen \thanks{Dep. of Mathematics, Iowa State University, Ames IA, USA. \texttt{\{wcchen, abparker, zerbib\}@iastate.edu}} \and
Alex Parker \footnotemark[1] \and
Shira Zerbib \footnotemark[1] \thanks{The research of S. Zerbib is supported by NSF CAREER award no. 2336239  and Simons Foundation award no. MP-TSM-00002629.}
}

\begin{document}
\maketitle
\begin{abstract}
 For an integer $s \ge 2$, a subset $S \subseteq [n]$ is {\em $s$-stable} if $\min \{j - i, n + i - j\}\ge s$ for every $i,j \in S$ with $i<j$. Denote the set of all $s$-stable subsets of size $k$ of $[n]$ by $\binom{[n]}{k}_{s\text{-stable}}$. Schrijver proved in 1978 that whenever $n\ge 2k$, the chromatic number of the Kneser graph $\mathrm{KG}\big( \binom{[n]}{k}_{2\text{-stable}}\big)$ is $n - 2k +2$. Generalizing this result, Meunier conjectured in 2011 that $\chi\left( \mathrm{KG}\big( \binom{[n]}{k}_{s\text{-stable}} \big) \right)= n - sk +s$ for all $n\ge sk$. This conjecture was previously proven for all even $s$ \cite{Chen-15}, for $s \ge 4$ and large enough $n$ \cite{Jonsson-12}, and  for $k=2$ \cite{DaneshMeun-21}. We prove the conjecture in the cases $s=3$ and $n$ large enough, or $k=s=3$. To this end, we prove versions of the Hilton-Milner theorem for $s$-stable sets. We also present a topological approach towards Meunier's conjecture.
\end{abstract}

\section{Introduction}

\subsection{Meunier's conjecture and main results}

Throughout the paper,  unless stated otherwise, $n,k,s$ are integers with $n \ge k \ge 2$ and $s\ge 2$. The set $\{1,2,\dots, n\}$ is denoted by $[n]$.  The {\em Kneser graph} $\KG(\mathcal{F})$ of a family of sets $\mathcal{F}$  is the graph whose vertex set is $\mathcal{F}$, and for $A, B \in \mathcal{F}$, $\{A,B\}$ forms an edge if $A \cap B = \emptyset$. When $\mathcal{F}$ is the collection $\binom{[n]}{k}$ of all $k$-subsets of $[n]$, the graph  $\KG(\mathcal{F})$ is denoted by $\KG(n,k)$. Let $\chi(G)$ be the chromatic number of the graph $G$. 

Proving a conjecture of Kneser~\cite{Kneser-55}, Lov\'asz~\cite{Lovasz-78} showed in 1978 that for $n\ge 2k$,  $\chi(\KG(n,k))=  n - 2k +2$. The upper bound  follows from the coloring  $c(A)= \min\{\min(A), n-2k +2\}$ for every set $A \in \binom{[n]}{k}$, while the lower bound was proven using the Borsuk-Ulam theorem.
 
For $i, j \in [n]$, with $i \leq j$,  let $d(i,j)= \min \{j - i, n + i - j\}$. That is, $d(i,j)$ is the distance between $i,j$ on the graph $C_n$ on vertex sets $[n]$, where the vertices on the cycle are ordered in the natural order. A set $A \in \binom{[n]}{k}$ is said to be  \emph{$s$-stable} if $d(i,j) \geq s$ for every $i,j \in A$. We  denote the collection of all $s$-stable $k$-subsets of $[n]$ by $\bstable{n}{k}{s}$, and the corresponding Kneser graph is denoted by $\kgstable{n}{k}{s}$. Note that $\kgstable{n}{k}{s}$ is an induced subgraph of $\KG(n,k)$.

Shortly after Lov\'asz' proof was published, Schrijver~\cite{Schrijver-78} showed that  $\chi(\kgstable{n}{k}{2})=\chi(\KG(n,k))$ for all  $n\ge 2k$. Furthermore, he showed that 
the graph $\kgstable{n}{k}{2}$ (now known as the {\em Schrijver graph}) is a vertex critical subgraph of  $\KG(n,k)$. 

 Observe that for all $s\ge 2$ and $n\ge sk$ one has  $\chi(\kgstable{n}{k}{s}) \leq n-sk +s$, as exhibited  by the coloring  $c(A) = \min\{ \min(A), n-sk +s\}$, for $A \in \bstable{n}{k}{s}$. Generalizing Schrijver's theorem, Meunier~\cite{Meunier-11} conjectured in 2011 that this upper bound is tight. 
\begin{conjecture}[Meunier \cite{Meunier-11}]
\label{conj:main}
    For all $s\ge 2$ and $n \geq sk$,
    $$\chi(\kgstable{n}{k}{s}) = n-sk +s.$$
\end{conjecture}

 The conjecture is trivial when $n=sk$. 
  In \cite{Meunier-11} Meunier proved the conjecture when $n = sk + 1$. 
Jonsson \cite{Jonsson-12} showed that Meunier's conjecture holds for all $s \geq 4$ when $n$ large enough (this manuscript is no longer available online). 

\begin{theorem}[Jonsson \cite{Jonsson-12}]
    \label{thm:jonsson-n-large}
    For all $k \geq 2$, $s \geq 4$, and $n \geq 2s(k-1)+2$,
    $$\chi(\kgstable{n}{k}{s}) = n - sk +s.$$
\end{theorem}

In 2015, P. Chen~\cite{Chen-15} resolved Meunier's conjecture for even $s$. 
\begin{theorem}[Chen \cite{Chen-15}]
    \label{thm:chen-even-s}
    If  $s \geq 4$ even and $n \geq sk$ then
    $$\chi(\kgstable{n}{k}{s}
    )= n - sk +s.$$
\end{theorem}

Moreover, the conjecture was proved for $k=2$. 
\begin{theorem}[Daneshpajouh, Meunier, and Mizrahi \cite{DaneshMeun-21}]
Let $n \ge 2s$. Then $$\chi(\kgstable{n}{2}{s}
    )= n - s.$$
\end{theorem}

Thus, except for the case $k = 2$, Meunier's conjecture for $s=3$ remained wide open. 
In this paper we resolve the conjecture for $s=3$ and  $n$ large enough, or when $k=s=3$. We prove:
\begin{theorem}
\label{thm:main:k>=4}
Let $k \geq 4$ and $n \geq k^3+3k^2.$ Then $$\chi(\KGk) = n-3k +3.$$  
\end{theorem}
\begin{theorem}
\label{thm:k=3}
    For $n \geq 9$, $$\chi(\kgstable{n}{3}{3}) = n - 6.$$
\end{theorem}

For our proofs we obtain versions of the Hilton-Milner theorems for $s$-stable sets.  

\subsection{Hilton-Milner theorems for $s$-stable sets}
For the proof of Theorems \ref{thm:k=3} and \ref{thm:main:k>=4} we wish to bound the maximum size of a color class in a proper coloring of $\kgstable{n}{k}{s}$. Note that a color class in a Kneser graph is a family of sets in which every two sets intersect. 

A family of sets in which every two sets intersect is called {\em intersecting}. 
An intersecting family $\cF$ is a {\em star} if $\bigcap_{F\in \mathcal{F}} F \neq \emptyset$.  Otherwise, we say that $\mathcal{F}$  is a {\em non-star}. If $\cF$ is a star, then an element $z\in \bigcap_{F\in \mathcal{F}} F$ is a {\em center} of $\cF$.
Two families of sets $\mathcal{A}$ and $\mathcal{B}$ are \emph{cross-intersecting} if $A \cap B \neq \emptyset$ for every $A \in \mathcal{A}$ and $B \in \mathcal{B}$. A family of sets is $k$-{\em uniform} if all its sets are of size $k$. 

In 1961, Erd\H{o}s, Ko, and Rado~\cite{EKR-61} proved a tight upper bound on the size of an intersecting $k$-uniform family.
\begin{theorem}[Erd\H{o}s-Ko-Rado \cite{EKR-61}] Let $n \geq 2k$.
    If $\mathcal{F} \subseteq \binom{[n]}{k}$ is an intersecting family,  then
    $$|\mathcal{F}| \leq \binom{n - 1}{k - 1}.$$
    Moreover, if $n > 2k$, equality holds if and only if $\mathcal{F}$ is the star consisting of all the $k$-subsets $[n]$ containing the element $i$, for some $i\in [n]$.
\end{theorem}

Talbot proved a similar theorem for $s$-stable sets. 
\begin{theorem}[Talbot \cite{Talbot-03}] Let $s \geq 2$ and  $n \geq sk$. 
    If $\mathcal{F} \subseteq \bstable{n}{k}{s}$ is an intersecting family, then
    $$|\mathcal{F}| \leq \binom{n - (s-1)k - 1}{k - 1}.$$
    Moreover, if $s = 2$ and $n \neq 2k + 2$ or $s \geq 3$, equality holds if and only if $\mathcal{F}$ is a star consisting of all the sets in $\bstable{n}{k}{s}$ containing the element $i$, for some $i\in [n]$.
\end{theorem}

The bound in the the Erd\H{o}s-Ko-Rado theorem can be improved if we are restricted to  non-star intersecting families. This is the content of the Hilton-Milner theorem \cite{HM-67}.
\begin{theorem}[Hilton-Milner \cite{HM-67}] Let $k \geq 2$ and $n \geq 2k + 1$. 
If $\mathcal{F} \subseteq \binom{[n]}{k}$ is a non-star intersecting family,  then
$$|\mathcal{F}| \leq \binom{n-1}{k-1} - \binom{n - k - 1}{k - 1} + 1.$$
\end{theorem}

Hilton and Milner also characterized the cases when equality holds. 

In this paper,  we prove versions of the Hilton-Milner theorem for  families of $s$-stable sets. To state the theorems, we  introduce some notation. 
For a set $X \in \binom{[n]}{3}$, let
$$\mathcal{F}_s(X) := \Bigl\{F \in \bstable{n}{3}{s} \Bigm| |F \cap X| \geq 2\Bigr\}.$$
Further, for $x \in [n]$ and $Y \in \bstable{n}{3}{s}$ with $x \not \in Y$, define
$$\mathcal{G}_s(x, Y) := \{ Y  \} \cup \left \{ F \in \bstable{n}{3}{s} \Bigm| x \in F, F \cap Y \neq \emptyset \right \}.$$


Given two families $\cF, \mathcal{G} \subseteq  \binom{[n]}{k}$, we say that $\cF$ is {\em isomorphic} to $\mathcal{G}$, and write $\cF\cong \mathcal{G}$, if there exists a permutation $\pi: [n] \to [n]$ such that $\{x_1, \dots, x_k\} \in \cF$ if and only $\{\pi(x_1), \dots, \pi(x_k)\} \in \mathcal{G}$. 

The next theorem is a ``Hilton-Milner theorem for $s$-stable $3$-sets". Namely, it gives a tight upper bound on the size of a $3$-uniform  non-star intersecting family of $s$-stable sets: 
\begin{theorem}
    \label{thm:stable-HM}
Let  $n \geq 4s+3$, and suppose $\mathcal{F} \subseteq \bstable{n}{3}{s}$ is a non-star intersecting family. Then $$|\mathcal{F}| \leq 3n - 10s + 2.$$
Moreover, equality holds if and only if $\cF\cong \cF_s(\{1,s+1,2s+1\})$ or $\cF\cong \mathcal{G}_s(1, \{s + 1, y, n - s + 1\})$ for some  $y \in [2s + 1, n - 2s + 1]$.    
\end{theorem}
With more careful analysis, the bound on $n$ in the theorem can probably be slightly improved.

We also prove a ``non-tight Hilton-Milner theorem for $3$-stable $k$-set". That is, we prove a bound on the size of a $k$-uniform  non-star intersecting family of $3$-stable sets. This bound is likely not tight. 

Theorem \ref{thm:tau-atleast-two} is primarily intended as a tool for proving Theorem \ref{thm:main:k>=4} rather than as an optimal extremal result in its own right. Unlike the case $k=3$, where the structure of intersecting families can be analyzed rather precisely, for larger uniformity we do not currently have a satisfactory description of extremal non-star intersecting families of 3-stable sets. Such a characterization appears substantially more difficult, as one must simultaneously control both the stability condition and the interaction between multiple potential centers. Fortunately, for our application to the chromatic number, an upper bound of the correct order of magnitude is sufficient, and the estimate given in Theorem \ref{thm:tau-atleast-two} is strong enough for this purpose.
\begin{theorem}
    \label{thm:tau-atleast-two}
    Let $k \geq 4$ and $n \geq k^3$. Let $\cF \subseteq \bstable{n}{k}{3}$ be a non-star intersecting family. Then
    $$|\cF| \leq \binom{n-2k-2}{k-2} + \binom{n-2}{k-1} - \binom{n-k-1}{k-1} + 1.$$
\end{theorem}

It seems plausible that the upper bound in Theorem \ref{thm:tau-atleast-two} is not optimal. Determining the exact maximum size of a non-star intersecting family of 3-stable $k$-sets remains an interesting open problem. More generally, one may ask for a Hilton–Milner theorem for $s$-stable $k$-uniform families for arbitrary $s$. Besides its intrinsic extremal interest, such a theorem could potentially lead to improvements in the range of $n$ for which Meunier's conjecture is known to hold. 

\subsection{Paper organization}
We start by proving Theorems \ref{thm:stable-HM} and \ref{thm:tau-atleast-two}  in Sections \ref{sec:HM1} and \ref{sec:HM2}, respectively.
Then in Section \ref{sec:tec} we prove a reduction lemma, which allows us to reduce Theorems \ref{thm:main:k>=4}  and \ref{thm:k=3} to proving an upper bound on the size of  non-star intersecting families. In Section \ref{sec:k=3} we then prove Theorems \ref{thm:k=3} and \ref{thm:main:k>=4}, in that order. 
 Finally, in Section \ref{sec:top} we discuss a new topological approach towards Meunier's conjecture. 
We employ this approach to prove the conjecture in the (already proven) case where $k=2$ and $n$ is divisible by $s-1$. This approach will likely not work in general, but we still find it interesting enough to present here.   

A central ingredient of our approach is the reduction lemma proved in Section \ref{sec:tec}. It shows that establishing the conjectured chromatic number reduces to obtaining sufficiently strong upper bounds on non-star intersecting families of stable sets. This separates the coloring problem from the extremal set-theoretic analysis and allows the proofs of Theorems \ref{thm:main:k>=4} and \ref{thm:k=3} to proceed in a unified manner.

\section{Proof of Theorem  \ref{thm:stable-HM}}
\label{sec:HM1}

For $x,y \in [n]$, define 
    $$\mathcal{F}_{x,y}^{k,s} := \left \{F \in \bstable{n}{k}{s} : \{x, y\} \subseteq F \right\}.$$
If $k,s$ are fixed, we often omit them from the notation

In this section we fix $k=3$ and write $ \mathcal{F}_{x,y}=\mathcal{F}_{x,y}^{3,s}$.
We begin by proving a few lemmas.    
\begin{lemma}
\label{lem:two-star}
    Let $s \geq 1$, $n \geq 4s -3$, and suppose $x, y \in [n]$. Then,
    $$|\mathcal{F}_{x,y} | = \begin{cases}
  n - 2s + 1 - \min\{2s - 1, d(x,y)\} & \text{if } d(x,y) \geq s, \\
  0 & \text{if } d(x,y) < s.
\end{cases}$$
  In particular,   $|\mathcal{F}_{x,y} | \le n-3s+1$.
\end{lemma}
\begin{proof}
    If $s = 1$,  the number of $3$-subsets of $[n]$ containing $x$ and $y$ is $n - 2$ and the result holds. Suppose  $s \geq 2$. Clearly, if $d(x,y) < s$, then no $s$-stable set can contain both $x$ and $y$. Therefore, we may assume $d(x,y) \geq s$. By relabeling, we may assume $x = 1$ and $d(x,y) = y - x = y-1$. Note that this assumption implies  $n+1-y \geq y-1$.  
    
    We wish to count the number of elements $z$ such that $\{x, y, z\} \not \in \bstable{n}{3}{s}$. 
    We first count the number of such elements $z$ in the interval $[y,1]$. 
    Let $A=[n - s + 2, n]$,  $B=[y+1, y + s - 1]$, and $C=\{x, y\}$ (see Figure \ref{fig:[n-s+2,n] and [y+1,y+s-1] are disjoint}). Note that no element from $A \cup B \cup C$  forms an $s$-stable set of size 3 with $x$ and $y$.

    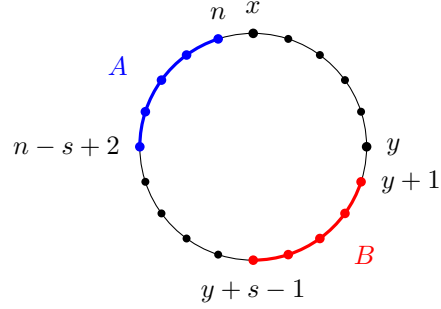
\begin{figure}[ht]
      \begin{center}
          
      \begin{tikzpicture}[scale=1.5]

    \def\r{1}
    
    \draw (0,0) circle (\r);
    
    \foreach \i in {1,...,20} {
        \coordinate (P\i) at ({90 - (\i-1)*18}:\r);
    }
    
    \foreach \i in {1,...,20} {
        \fill (P\i) circle (1pt);
    }
    \foreach \i in {16,...,20} {
        \fill[blue] (P\i) circle (1.2pt);
    }
    \foreach \i in {7,...,11} {
        \fill[red] (P\i) circle (1.2pt);
    }
    \fill[black] (P1) circle (1.2pt);
    \fill[black] (P6) circle (1.2pt);
    
    \node[above, yshift=4pt] at (P1) {$x$};
    
    \node[right,xshift = 4pt] at (P6) {$y$};
    \node[right,xshift = 4pt] at (P7) {$y+1$};
    
    \node[below, yshift = -4pt] at (P11) {$y+s-1$};
    
    \node[above, yshift=4pt] at (P20) {$n$};
    
    \node[left, xshift = -4pt] at (P16) {$n-s+2$};
    
    \node[blue,left] at ({145}:1.25) {$A$};
    
    
    \node[red,right] at ({-50}:1.25) {$B$};
    
    \draw[blue, very thick]
    ({90 - (20-1)*18}:1)
    arc[start angle={90 - (20-1)*18},
        end angle={90 - (15)*18},
        radius=1];
    
    \draw[red, very thick]
    ({90 - (11-1)*18}:1)
    arc[start angle={90 - (11-1)*18},
        end angle={90 - (6)*18},
        radius=1];
    \end{tikzpicture}
    \caption{The sets $A$ and $B$  in the proof of Lemma \ref{lem:two-star}.}
    \label{fig:[n-s+2,n] and [y+1,y+s-1] are disjoint}
    \end{center}
\end{figure}
    
We claim that the sets $A,B,C$ are pairwise disjoint. Indeed, it is clear that $C$ is disjoint from $A$ and $B$. If $A$ intersects $B$ then $y+s-1 \geq n-s+2$. Thus 
\begin{equation}
\label{eq:1}
    n \leq y+2s-3.
\end{equation}
 Since by our assumption $n+1-y \geq y-1$, we have $$y \leq \frac{n}{2}+1.$$ Thus together with (\ref{eq:1}) we get $n \leq 4s-4,$ a contradiction. Therefore we have, $|A\cup B \cup C| =2s$, showing that the number of elements $z$ in the interval $[y,1]$ such that $\{x, y, z\} \not \in \bstable{n}{3}{s}$.

  We now wish to count the number of such elements $z$ in the interval $[2,y-1]$. 
  If $d(x,y) \geq 2s - 1$, then there are at least $2s - 2$ elements in that interval and therefore, no element from $[2, s] \cup [y-s+1, y - 1]$  forms an $s$-stable set of size 3 with $x$ and $y$. 
  If $d(x,y) \le 2s - 2$, there are $d(x,y) - 1$ elements that may not form an $s$-stable $3$-set with $x$ and $y$. Putting together, there are $$2s + \min\{2s - 2, d(x,y) - 1\} = 2s - 1 + \min\{2s - 1, d(x,y)\}$$ elements in $[n]$ that do not form an $s$-stable 3-set with $x,y$, and the lemma follows. 
\end{proof}

\begin{lemma}
    \label{lem:min-dist-HM1}
    Let $s \geq 1$, $n \geq 4s + 1$, and suppose $\{x, y, z\} \in \bstable{n}{3}{s}$. Then, $$d(x,y) + d(x,z) + d(y,z) \geq 4s$$ and equality holds if and only if $\{d(x,y), d(x,z), d(y,z)\} = \{s,s,2s\}$.
\end{lemma}
\begin{proof}
    By relabeling, we may assume  $x \leq y \leq z$ and  $s \leq d(x,y) \leq d(x,z) \leq d(y,z)$. Then, the shortest $(x,y)$-path of $C_n$ does not contain $z$ and the shortest $(x,z)$-path of $C_n$ does not contain $y$. Suppose the shortest $(y,z)$-path contains $x$. Then, $d(y,z) = d(y,x) + d(x,z) \geq 2s$ and therefore, the sum of distances is at least $4s$ with equality if and only if $d(y,x) = d(x,z) = s$ (and therefore, $d(y,z) = 2s$). On the other hand, if the shortest $(y,z)$-path does not contain $x$, then $d(y,z) = n - d(y,x) - d(x,z)$. In this case, we get that the sum of distances is equal to $n$, which we assumed to be at least $4s + 1$.
\end{proof}

In the following two propositions we show that the largest  families of the form $\mathcal{F}_s(X)$  and $\mathcal{G}_s(x,Y)$ are the those where equality holds in the statement of Theorem~\ref{thm:stable-HM}.

\begin{proposition}
    \label{prop:largest-HM1}
    Let $s \geq 1$, $n \geq 4s + 1$, and suppose $\{x,y,z\} \in \binom{[n]}{3}$. Then, we have
    $$|\mathcal{F}_s(\{x,y,z\})| \leq 3n - 10s + 2.$$
    Furthermore, if $s = 1$, equality holds for any choice of $\{x,y,z\}$. If $s \geq 2$, equality holds if and only if $\{x,y,z\}$ is isomorphic to $\{1,s+1,2s+1\}$. 
\end{proposition}
\begin{proof}
    When $s = 1$,  observe that for any pair of elements in $\{x,y,z\}$, the number of subsets of $[n]$ of size 3 containing this pair is $n-2$. However, the set $\{x,y,z\}$ is counted once for each pair. Therefore,
    $$|\mathcal{F}_1(\{x,y,z\})| = 3(n-2) - 2 = 3n - 8.$$

    Suppose $s \geq 2$. 
    By Lemma \ref{lem:two-star} we have  $|\mathcal{F}_{i,j} | \le n-3s+1$. 
    Moreover, $\mathcal{F}_s(\{x,y,z\}) = \mathcal{F}_{x,y} \cup \mathcal{F}_{x,z} \cup \mathcal{F}_{y,z}$. Therefore, if  some pair of elements in  $\{x,y,z\}$ has distance less than $s$, then
    $$|\mathcal{F}_s(\{x,y,z\})| \leq 2(n-3s+1) = 2n-6s + 2 < 3n-10s+2.$$ The last inequality holds since we assumed $n \geq 4s + 1$.
    
    Therefore, we may assume that all pairwise distances between $x,y,z$ are at least $s$. In this case we have $|\mathcal{F}_s(\{x,y,z\})| = |\mathcal{F}_{x,y}| + |\mathcal{F}_{x,z}| + |\mathcal{F}_{y,z}| - 2$, and therefore, by Lemma~\ref{lem:two-star}, 
    \begin{align*}
        |\mathcal{F}_s(\{x,y,z\})| =   3n& - 6s + 1 \\
        &- (\min\{2s - 1, d(x,y)\} + \min\{2s - 1, d(x,z)\} + \min\{2s - 1, d(y,z)\}).
    \end{align*}

Write $D:=\min\{2s - 1, d(x,y)\} + \min\{2s - 1, d(x,z)\} + \min\{2s - 1, d(y,z)\}$. To prove the proposition, we show that $D \ge 4s-1$,  
 with equality if and only if $\{d(x,y), d(x,z), d(y,z)\} = \{s,s,2s\}$. 
 Indeed, if all the pairwise distances are less than $2s-1$, then by Lemma~\ref{lem:min-dist-HM1} we have $D \ge 4s> 4s-1. $ 
 If all the distances are at least $2s-1$, then 
 $D = 3(2s - 1) = 6s - 3 > 4s-1$. 
 If exactly two of the distances are at least $2s-1$,
then $D \ge s + 2(2s-1) = 5s - 2 > 4s-1$. 
 Finally, if exactly one of the distances is at least $2s-1$, then     $D \ge s+s + 2s-1 \geq 4s-1$, and equality holds if and only if two of the distances are exactly $s$, and therefore the third is $2s$.  
\end{proof}

\begin{proposition}
    \label{prop:largest-HM2}
    Let $s \geq 1$, $n \geq 4s + 1$,  $x \in [n]$, and $Y \in \bstable{n}{3}{s}$ with $x \not \in Y$. Then, 
    $$|\mathcal{G}_s(x,Y)| \leq 3n - 10s + 2.$$
    Furthermore, if $s = 1$, equality holds for any choice of $x \in [n]$, and $Y \in \bstable{n}{3}{s}$ with $x \not \in Y$. If $s \geq 2$, equality holds if and only if there exists $y \in [2s+1, n - 2s + 1]$ such that $\mathcal{G}_s(x,Y)$ is isomorphic to $\mathcal{G}_s(1, \{s+1, y, n - s + 1\})$. 
\end{proposition}
\begin{proof}
    Suppose first that $s = 1$. The number of $3$-subsets of $[n]$ that contain $x$ but do not contain any element of $Y$ is $\binom{n-4}{2}$. Therefore, the number of $3$-subsets of $[n]$ that contain $x$ and contain at least one element of $Y$ is $\binom{n - 1}{2} - \binom{n-4}{2}$. Together with $Y$, this gives
    $$|\mathcal{G}_1(x,Y)| = \binom{n - 1}{2} - \binom{n-4}{2} + 1 = 3n - 8,$$ As needed.
    
    Next, suppose $s \geq 2$. Write $Y=\{y_1,y_2,y_3\}$. Observe that since $Y$ is a $s$-stable, we must have $d(x,y_i)\ge s$ for some $i\in [3]$. Without loss of generality, $d(x,y_3)\ge s$.

 If $d(x,y_i)<s$ for both $i=1$ and $i=2$, then $\mathcal{G}_s(x,Y) =  \cF_{x,y_3} \cup \{Y\}$, and thus by Lemma \ref{lem:two-star} we have $$|\mathcal{G}_s(x,Y)| \leq n-3s+1+1 < 3n-10s+2.$$
    
    If $d(x,y_i)<s$ for exactly one of $i\in\{1,2\}$, say $i=1$, then $\mathcal{G}_s(x,Y) = \cF_{x,y_2} \cup \cF_{x,y_3} \cup \{Y\}$, and by Lemma \ref{lem:two-star} we have 
    $$|\mathcal{G}_s(x,Y)| \leq 2(n-3s+1)-1+1 = 2n-6s+2 < 3n-10s+2.$$
We subtracted $1$ since the set $\{x,y_2,y_3\}$ belongs to both  $\cF_{x,y_2}$ and  $\cF_{x,y_3}$.

    Therefore, we may assume $d(x,y_i)\ge s$ for all $i\in [3]$. Then 
     $\mathcal{G}_s(x,Y) = \mathcal{F}_{x,y_1} \cup \mathcal{F}_{x,y_2} \cup \mathcal{F}_{x,y_3} \cup \{Y\}$, and since the sets $\{x,y_1,y_2\}, \{x,y_1,y_3\}, \{x,y_2,y_3\}$ are counted twice in this union, we get by Lemma \ref{lem:two-star},
    \begin{align*}
    |\mathcal{G}_s(x,Y)| &= |\mathcal{F}_{x,y_1}| + |\mathcal{F}_{x,y_2}| + |\mathcal{F}_{x,y_3}| - 2\\
        &= 3n - 6s + 1 - (\min\{2s - 1, d(x,y_1)\} + \min\{2s - 1, d(x,y_2)\} + \min\{2s - 1, d(x,y_3)\}).
    \end{align*}

Write $D=\min\{2s - 1, d(x,y_1)\} + \min\{2s - 1, d(x,y_2)\} + \min\{2s - 1, d(x,y_3)\}.$
We show that  $D\ge 4s - 1$ with equality if and only if there exists $y_i, y_j \in Y$ such that $d(x,y_i) = d(x,y_j) = s$.

By relabeling, we may assume $x \leq y_1 \leq y_2 \leq y_3$. Observe that since all distances are at least $s$, we must have  $d(x,y_2) \ge 2s$.  
If all  distances $d(x,y_i)$  are at least $2s-1$, then $$D= 3(2s - 1) = 6s - 3 > 4s-1.$$ 
If exactly two of the distances are at least $2s-1$, then  $$D\ge s + 2(2s-1) = 5s - 2 > 4s-1.$$ Finally, if both $d(x,y_1)<2s-1$ and $d(x,y_3)<2s-1$, then $D=d(x,y_1) + d(x,y_3) + 2s-1 \geq 4s-1$ and equality holds if and only if $d(x,y_1) = d(x,y_3) = s$. In this case, we have that
$\mathcal{G}_s(x, Y)$ is isomorphic to $\mathcal{G}_s(1, \{s+1, y, n-s+1\})$ for some $y \in [2s + 1, n - 2s + 1]$.
\end{proof}

The {\em matching number} $\nu(\cF)$ of a family of sets $\cF$  is the largest size of a subfamily of pairwise disjoint sets in  $\cF$. 
The {\em covering number} $\tau(\cF)$ is the smallest size of a {\em cover}, namely, a set of vertices $C$ that intersects every set in $\cF$. Note that if $\cF$ is a non-star, then $\tau(\cF) \ge 2$.

\begin{lemma}
    \label{lem:non-star-graphs}
    Let $A, B$ be non-empty, cross-intersecting families of sets of size 2, such that neither $A$ nor $B$ is a star. Then, $|A| + |B| \leq 6$.
\end{lemma}
\begin{proof}
    Observe that since the families are 2-uniform, cross-intersecting, and non-empty, we have $\nu(A) \leq 2$ and $\nu(B) \leq 2$. If $\nu(A) = 1$, then $A$ is a non-star intersecting graph, so it must be a triangle, and $|A|=3$.  Since $B$ is not a star and $A, B$ are cross-intersecting, we must have that $A = B$ and therefore, $|A| + |B| = 6$. A symmetric argument holds when $\nu(B) = 1$. Therefore, we may assume $\nu(A) = \nu(B) = 2$. 
    
    Let $M = \{e_1=\{v_1, v_2\}, e_2=\{v_3, v_4\}\}$ be a maximum matching of $A$. Then, since $A, B$ are cross-intersecting, 
    we have that $B$ is a bipartite graph with sides $e_1$ and $e_2$. 
    Because $\nu(B) = 2$, after relabeling, we may assume $\{e_3=\{v_1, v_3\}, e_4=\{v_2, v_4\}\} \subseteq B$. Thus $A$ is a  bipartite graph with sides $e_3$ and $e_4$. 
    If each of $A$ and $B$ has at most 3 edges then we are done. Otherwise, one of $A$ or $B$ is the complete bipartite graph $K_{2,2}$. But then the other family cannot have more than two edges without violating the cross-intersection property. So again we have, 
 $|A| + |B| = 6$.
\end{proof}

We  will also use the following result by Erd\H{o}s and Lov\'asz~\cite{erdos-lovasz-75}:

\begin{theorem}[Erd\H{o}s and Lov\'asz \cite{erdos-lovasz-75}]
    \label{thm:cover>2}
    Let $\mathcal{F} \subseteq \binom{[n]}{3}$ be an intersecting family with $\tau(\mathcal{F}) \geq 3$. Then, $|\mathcal{F}| \leq 10$.
\end{theorem}

We are now ready to prove Theorem~\ref{thm:stable-HM}.

\begin{proof}[Proof of Theorem~\ref{thm:stable-HM}]
Let
$\mathcal{F} \subseteq \bstable{n}{3}{s}$ be a non-star intersecting family.
    By Theorem~\ref{thm:cover>2}, if $\tau(\cF)\ge 3$ then  $|\cF|\le 10$. 
    Note that $10 < 3n - 10s + 2$ whenever $n \geq 4s + 3$. Therefore, we may assume that $\tau(\mathcal{F}) = 2$.

    Let $\{x, y\}$ be a cover of $\mathcal{F}$. Then $\cF$ can be written as a disjoint union $\mathcal{F} = \mathcal{H}_{x,y} \sqcup\mathcal{H}_x \sqcup \mathcal{H}_y,$
    where $\mathcal{H}_{x,y} = \{S\in \cF : \{x,y\}\subset S\}$, $\mathcal{H}_{x} = \{S\in \cF :  x\in S, y\notin S\}$, and $\mathcal{H}_{y} = \{S\in \cF :  y\in S, x\notin S\}$. Note that  $\tau(\mathcal{F}) = 2$ implies $\mathcal{H}_x \neq \emptyset$ and $\mathcal{H}_y \neq \emptyset$.
    Define the following families:
    $$\mathcal{H}_x(x) = \{S - \{x\} : S \in \mathcal{H}_x\}, \quad \mathcal{H}_y(y) = \{S - \{y\} : S \in \mathcal{H}_y\}.$$
Since $\mathcal{F}$ is $3$-uniform and intersecting, $\mathcal{H}_x(x)$ and $\mathcal{H}_y(y)$ are cross-intersecting  graphs. Split into cases. 
    
    \textbf{Case 1.} Neither $\mathcal{H}_x(x)$ nor $\mathcal{H}_y(y)$ is a star. In this case, by Lemma~\ref{lem:non-star-graphs}, $|\mathcal{H}_x(x)| + |\mathcal{H}_y(y)| \leq 6$. Now, if $d(x,y) < s$, $\mathcal{H}_{x,y} = \emptyset$. Otherwise, by Lemma~\ref{lem:two-star}, we have $|\mathcal{H}_{x,y}| \le n - 3s + 1$. This gives 
$$|\mathcal{F}| = |\mathcal{H}_{x,y}| + |\mathcal{H}_x(x)| + |\mathcal{H}_y(y)| \le  n - 3s + 7 < 3n - 10s + 2.$$
    The last inequality holds for all $s \geq 1$ and $n \geq 4s+3$.
    
    \textbf{Case 2.} Both $\mathcal{H}_x(x)$ and  $\mathcal{H}_y(y)$ are stars. 
    
    Assume first that the stars $\mathcal{H}_x(x)$ and  $\mathcal{H}_y(y)$ have different centers, say $\mathcal{H}_x(x)$ has center $z$, $\mathcal{H}_y(y)$ has center $w$, such that $z$ is not a star for $\mathcal{H}_y(y)$ and $w$ is not a star for $\mathcal{H}_x(x)$. Then, every set in $\mathcal{H}_x$ is of the form $\{x, z, u\}$ such that $u \notin \{x,y,z\}$, and every set in $\mathcal{H}_y$ is of the form $\{y, w, v\}$ with  $v \notin \{x,y,w\}$. 
    
    We claim that either $|\mathcal{H}_x(x)| + |\mathcal{H}_y(y)| \leq 4$ or at most one family has size greater than $1$. Indeed, suppose $|\mathcal{H}_x(x)| + |\mathcal{H}_y(y)| > 4$ and both families have size greater than $1$. Without loss of generality, we may assume $|\mathcal{H}_x(x)| \geq 3$. Let $\{x,z,u_1\}, \{x,z,u_2\}, \{x,z,u_3\} \in \mathcal{H}_x$ and $\{y,w,v_1\}, \{y,w,v_2\} \in \mathcal{H}_y$. Then, to maintain the cross-intersection property between $\{x, z, u_1\}$ and the two sets from $\mathcal{H}_y$, one would need $u_1 = w$ or $\{u_1, z\} = \{v_1, v_2\}$. In the first case, if we next consider the need for the cross-intersection property between $\{x,z,u_2\}$ and the two sets from $\mathcal{H}_y$, we would conclude that $\{u_2,z\} = \{v_1,v_2\}$. But then, $\{x,z,u_3\}$ is disjoint from one of the two sets from $\mathcal{H}_y$, a contradiction. In the second case,  both $\{x,z,u_2\}$ and $\{x,z,u_3\}$ are disjoint from one of the two sets from $\mathcal{H}_y$, again a contradiction.

    Now, if $|\mathcal{H}_x(x)| + |\mathcal{H}_y(y)| \leq 4$, we get 
    $$|\mathcal{F}| = |\mathcal{H}_{x,y}| + |\mathcal{H}_x(x)| + |\mathcal{H}_y(y)| \le  n - 3s + 5 < 3n - 10s + 2.$$
     On the other hand, if $|\mathcal{H}_x(x)| + |\mathcal{H}_y(y)| > 4$, then one of the families is of size 1. Without loss of generality, suppose $\mathcal{H}_x = \{\{x,z,u\}\}$. Then, to maintain the cross intersection property, we must have $u = w$ and thus $\mathcal{H}_x(x) = \{\{z, w\}\}$. But this means that $w$ is  a center  for $\mathcal{H}_x(x)$, a contradiction to our assumption.

    If $\mathcal{H}_x(x)$ and  $\mathcal{H}_y(y)$ have a mutual center, say $w$, then every element of $\mathcal{F}$ contains at least two elements from $\{x,y,w\}$. Therefore, $\mathcal{F} \subseteq \mathcal{F}_s(\{x,y,w\})$, and thus by Proposition \ref{prop:largest-HM1} we are done. 

    \textbf{Case 3.} One family is a star and the other family is not a star. Without loss of generality, suppose $\mathcal{H}_x(x)$ is a star with center $z$ and $\mathcal{H}_y(y)$ is not a star. Write $$\mathcal{H}_{y}(y)= \mathcal{H}_{y,z}(y) \cup \mathcal{H}_{y, \overline{z}}(y),$$ where $\mathcal{H}_{y,z}(y)$ is the collection of all sets in  $\mathcal{H}_y(y)$  containing $z$ and $\mathcal{H}_{y, \overline{z}}(y)$ is the collection of all sets in $\mathcal{H}_y(y)$ not containing $z$. Since $\mathcal{H}_y(y)$ is not a star, $\mathcal{H}_{y, \overline{z}}(y) \neq \emptyset$. If $|\mathcal{H}_x(x)| \geq 3$, then no set of $\mathcal{H}_{y, \overline{z}}(y)$ can intersect every element of $\mathcal{H}_x(x)$. Thus $|\mathcal{H}_x(x)| \le  2$. 
    
    Assume first that $|\mathcal{H}_x(x)| = 2$, say $\mathcal{H}_x(x) = \{\{z, u_1\}, \{z, u_2\}\}$. Then we must have $\mathcal{H}_{y, \overline{z}}(y) = \{\{u_1, u_2\}\}$ and $|\mathcal{H}_{y, \overline{z}}(y)|=1$. 
    If both $d(x,y) \geq s$ and $d(y,z) \geq s$, then the pairwise distance between any two elements in $\{x,y,z,u_1,u_2\}$ is at least $s$, which implies that $n \geq 5s.$
     By Lemma \ref{lem:two-star} we obtain
    \begin{align*}
        |\mathcal{F}| &= |\mathcal{H}_{x,y}| + |\mathcal{H}_{y,z} (y)|+|\mathcal{H}_{y,\bar z}(y)|+|\mathcal{H}_x(x)| \\
        &\le |\mathcal{H}_{x,y} \cup \mathcal{H}_{y,z}|+|\mathcal{H}_{y,\bar z}(y)|+|\mathcal{H}_x(x)|\\ 
        &\leq 2(n-3s+1) - 1 + 1 +2 = 2n-6s+4 < 3n-10s+2.
    \end{align*}
    We subtract $1$  since $\{x,y,z\}$ was counted twice.
    On the other hand, if one of $d(x,y), d(y,z)$ is less than $s$, 
    then for all $n \geq 4s+3$ we have
    $$|\mathcal{F}| \leq (n-3s+1)+3 = n-3s+4 < 3n-10s+2.$$
     Finally, assume $|\mathcal{H}_x(x)| = 1$, say $\mathcal{H}_x(x) = \{\{z, u\}\}$. Then, every element of $\mathcal{H}_y$ must contain either $z$ or $u$. In this case, $\mathcal{F} \subseteq \mathcal{G}_s(y, \{x, z, u\})$. 

    We have shown that for $n \geq 4s + 3$, $|\mathcal{F}| \leq 3n - 10s + 2$. Moreover, if $\mathcal{F}$ is not isomorphic to  $\mathcal{F}_s(X)$ or  $\mathcal{G}_s(y,X)$ for some $X$ and $y$, then $|\mathcal{F}| < 3n - 10s + 2$. This completes the proof of the theorem. 
\end{proof}

\section{Proof of Theorem  \ref{thm:tau-atleast-two}}\label{sec:HM2}

 In the proof, we will use the following two theorems by Frankl and Kupavskii~\cite{Frankl-Kupavskii-23} and Frankl and Wang~\cite{frankl-wang-2024}.

\begin{theorem}[Frankl and Kupavskii, Theorem 8 \cite{Frankl-Kupavskii-23}]
    \label{thm:cover-number-bound}
    Let $k \geq 2$ and $n \geq k^2$. If $\cF \subseteq \binom{[n]}{k}$ is an intersecting family with $\tau(\cF) = \tau$, then 
    $$|\cF| \leq k^{\tau}\binom{n-\tau}{k-\tau}.$$
\end{theorem}
Observe that $k^{\tau}\binom{n-\tau}{k-\tau} \geq k^{\tau+1}\binom{n-\tau-1}{k-\tau-1}$ if $k \geq 2$ and $n \geq k^2.$ Therefore, if $\cF \subset \binom{[n]}{k}$ is intersecting with $\tau(\cF) \geq 3,$ then $|\cF| \leq k^3\binom{n-3}{k-3}.$



\begin{theorem}[Frankl and Wang \cite{frankl-wang-2024}]
\label{thm:Frankl-Wang-Cross-intersecting-estimate}
    Let $\mathcal{A} \subset \binom{[n]}{a}$ and $\mathcal{B} \subset \binom{[n]}{b}$ be nonempty cross-intersecting families with $n \geq a+ b$ and $a \leq b.$ Then $$|\mathcal{A}|+ |\mathcal{B}| \leq \binom{n}{b}-\binom{n-a}{b}+1.$$
\end{theorem}

We will also need a few lemmas. For the rest of this section, fix  $s=3$ and 
for $x,y \in [n]$ write $\cF_{x,y} = \cF_{x,y}^{k,3}.$

\begin{lemma}
    Let $y\in [n]$ such that $5 \le y\le \frac{n}{2}$. Then $|\cF_{1,y}| \leq |\cF_{1,4}|$. 
\end{lemma}
\begin{proof} 
     Write $\cF_{1,y,4} = \{S \in \cF_{1,y} : 4 \in S\}$,  $\cF_{1,y,\overline{4}} = \cF_{1,y} \setminus \cF_{1,y,4}$, and 
      $\cF_{1,\overline{y},4} = \{S \in \cF_{1,4} : y \notin S\}$. Since $\cF_{1,y,4} = \cF_{1,4} \cap \cF_{1,y}$, the lemma will follow if we exhibit an injection $\cF_{1,y,\overline{4}} \to \cF_{1,\overline{y}, 4}$. 

We define a map $g:\cF_{1,y,\overline{4}} \to \cF_{1,\overline{y}, 4}$ as follows.
    Let $S \in \cF_{1,y,\overline{4}}$ and write $S = \{1,y\} \sqcup B \sqcup C$, where 
    $B = S \cap \{5,6,\dots, y-3\}$ and $C = S \setminus(\{1,y\} \cup B).$
    Let $f(B) = B + 2 = \{b + 2: b \in B\}$ and define $g(S) = \{1,4\} \cup f(B) \cup C$.
    
    We claim that $g(S) \in \cF_{1,\overline{y},4}$. Indeed, the smallest element of $B$ is $5$ and so, the smallest element of $f(B)$ is $7$. Also, since $y \in S$, the largest element of $B$ is at most $y-3$, and thus the largest element of $f(B)$ is at most $y-1$. Finally, the smallest element of $C$ is at least $y + 3$. Therefore, $|g(S)|=|S|$ and  $g(S)$ is 3-stable. 

    Finally, we claim that $g$ is injective. Indeed, let $S_1, S_2 \in \cF_{1,y,\overline{4}}$ with $S_1 \neq S_2$. For $i=1,2$ write $S_i = \{1,y\} \sqcup B_i \sqcup C_i$ as above. Suppose $g(S_1) = g(S_2)$. Since $f(B_1), f(B_2) \subseteq \{7, 8, \dots, y-1\}$ and $C_1, C_2 \subseteq \{y + 3, y + 4, \dots, n - 2\}$, we have that $f(B_1) = f(B_2)$ and $C_1 = C_2$. But then, $B_1 = B_2$, which implies $S_1 = S_2$, a contradiction.
\end{proof}

\begin{lemma}
    \label{lem:two-intersecting}
    Let $k \geq 4$ and $n \geq 3k$. Then for all $x,y \in [n]$ we have
    $$|\cF_{x,y}| \leq \binom{n-2k-2}{k-2}.$$
    Furthermore, when $d(x,y) = 3$ equality holds.
\end{lemma}
\begin{proof}
    Let $x,y \in [n]$. By relabeling, we may assume $x = 1$ and $y \leq \frac{n}{2}$. If $d(1,y) < 3$, then no $3$-stable set can contain both $1$ and $y$, and thus $|\cF_{1,y}| = 0$. Suppose  $d(1,y) \ge 3$. By the previous lemma, $|\cF_{1,y}| \leq |\cF_{1,4}|$. Therefore,  it is enough to show that $|\cF_{1,4}|=\binom{n-2k-2}{k-2}$.

We have
    \begin{align*}
        |\cF_{1,4}|&= |\{(a_1,a_2,\dots,a_{k-2}):7 \leq a_1<\cdots<a_{k-2} \leq n-2, ~a_{i}-a_{i-1} \geq 3 \text{ for }2\leq i \leq k-2\}| \\
        &=|\{(a_1', a_2',\dots,a_{k-2}'):1 \leq a_1'<\cdots<a_{k-2}' \leq n-8, ~a_{i}'-a_{i-1}' \geq 3 \text{ for }2\leq i \leq k-2\}|\\
        &=|\{(c_1,\dots,c_{k-1}):c_1,c_{k-1} \geq 0, ~c_i \geq 3 \text{ for }2 \leq i \leq k-2, ~\sum_{i=1}^{k-1}c_i = n-9\}| \\
        &=|\{(c_1',\dots,c_{k-1}'):c_i' \geq 0 \text{ for }1 \leq i \leq k-1, ~\sum_{i=1}^{k-1}c_i' = n-9-3(k-3)=n-3k\}| \\
        &= \binom{n-2k-2}{k-2},
    \end{align*}
    As needed. The third equality follows by letting $c_1=a_1'-1$, $c_{k-1}=n-8-a_{k-2}'$, and $c_i=a_{i}'-a_{i-1}'$ for all $2\le i\le k-2$. 
\end{proof}

\begin{proposition}
    \label{prop:tau-is-two}
    Let $k \geq 4$ and $n \geq 3k$. Suppose $\cF \subseteq \bstable{n}{k}{3}$ is an intersecting family with $\tau(\cF) = 2$. Then
    $$|\cF| \leq \binom{n-2k-2}{k-2} + \binom{n-2}{k-1} - \binom{n-k-1}{k-1} + 1.$$
\end{proposition}
\begin{proof}
    Let $x,y\in [n]$ be a vertex cover of $\cF$. As before, write $\cF = \cH_{x,y} \sqcup \cH_x \sqcup \cH_y$, where
    $$\cH_{x,y} = \{S \in \cF :x,y \in S\}, \quad
        \cH_x =\{S \in \cF: x \in S, y \not\in S\}, \quad
        \cH_y = \{S \in \cF: x \not\in S, y \in S\}.$$
    Note that by Lemma~\ref{lem:two-intersecting}, $|\cH_{x,y}| \leq |\cF_{x,y}| \leq \binom{n-2k-2}{k-2}.$ 
    
    We would like to bound $|\cH_x|+|\cH_y|$. Since $\cF$ is intersecting and $x\neq y$, $\cH_x$ and $\cH_y$ are cross-intersecting. Moreover, both $\cH_x$ and $\cH_y$ are non-empty, for otherwise $\tau(\cF) = 1$. Define 
    $$\cH_x'=\{A-\{x\}:A \in \cH_x\}, \quad
    \cH_y'=\{A-\{y\}:A \in \cH_y\}.$$
    Then $\cH_x'\subset \binom{[n] \setminus \{x,y\}}{k-1}$ and $\cH_y'\subset \binom{[n] \setminus \{x,y\}}{k-1}$ are non-empty and cross-intersecting. By Theorem \ref{thm:Frankl-Wang-Cross-intersecting-estimate}  we have $$|\cH_x|+|\cH_y|=|\cH_x'|+|\cH_y'| \leq \binom{n-2}{k-1}-\binom{n-k-1}{k-1}+1.$$ 
    
    Therefore, $$|\cF| = |\cH_{x,y}|+|\cH_x|+|\cH_y| \leq \binom{n-2k-2}{k-2} + \binom{n-2}{k-1}-\binom{n-k-1}{k-1}+1.$$
\end{proof}

We need one more lemma, whose proof is given in Appendix 1. 

\begin{lemma}
    \label{lem:tau-three-less-than}
    For all $k \geq 4$ and $n \geq k^3$, 
    $$k^3 \binom{n-3}{k-3} \leq \binom{n-2}{k-1} - \binom{n-k-1}{k-1}.$$
\end{lemma}

We are now ready to prove Theorem \ref{thm:tau-atleast-two}.
\begin{proof}[Proof of Theorem \ref{thm:tau-atleast-two}]
    If $\tau(\cF) \geq 3$, then by Theorem~\ref{thm:cover-number-bound} and Lemma~\ref{lem:tau-three-less-than} we have  $$|\cF| \leq k^3 \binom{n-3}{k-3}\leq \binom{n-2}{k-1} - \binom{n-k-1}{k-1} \leq \binom{n-2k-2}{k-2} + \binom{n-2}{k-1} - \binom{n-k-1}{k-1} + 1.$$
    If $\tau(\cF) = 2$, then by Proposition~\ref{prop:tau-is-two} we have
    $$|\cF| \leq \binom{n-2k-2}{k-2} + \binom{n-2}{k-1} - \binom{n-k-1}{k-1} + 1.$$
\end{proof}

\section{A reduction lemma}
\label{sec:tec}
In this section, we prove  the following: 
\begin{lemma}
    \label{thm:reduction}
    Let $k \geq 3$ and $n \geq 3k$. If every non-star intersecting family $\cF \subset \bstable{n}{k}{3}$ satisfies
    $$|\mathcal{F}| < \frac{n}{k(n-3k+2)}\binom{n-2k-1}{k-1},$$ then $\chi(\kgstable{n}{k}{3}) = n - 3k+3.$
\end{lemma}

 We begin by recalling two results we will use. For $X \in \binom{[n]}{k}$, let $X(1) < X(2) < \cdots< X(k)$ be the ordered  elements of $X$. 
Given a $k$-vector $\vec{s} = (s_1,\dots,s_k)$ of positive integers, $X$ is said to be {\em $\vec{s}$-stable} if for $1 \leq i \leq k-1$, $X(i+1) - X(i) \geq s_i$ and $n + X(1) - X(k) \geq s_k$. We denote the set of all  $\vec{s}$-stable $k$-sets by $\bstable{n}{k}{\vec{s}}$. Observe that when $\vec{s}=(s,s,\dots, s)$, an $\vec{s}$-stable set is $s$-stable. For the proof of the lemma, we will use the following theorem from \cite{Daneshpajouh-Osztenyi-2021}. 
\begin{theorem}[Daneshpajouh and Osztényi, \cite{Daneshpajouh-Osztenyi-2021}]
    \label{cor:almost-3stable}
    Let $n,k$ be integers with $k \geq 2$ and $n \geq 3k-1$. Consider the $k$-vector $\vec{s} = (3, \dots, 3,2)$. Then,
    $$\chi\left(\KG\left( \bstable{n}{k}{\vec{s}} \right)\right) = n - 3k+3.$$
\end{theorem}

Observe that for $\vec{s}$ as in the theorem, we have $\bstable{n}{k}{3} \subseteq \bstable{n}{k}{\vec{s}}$.

We will also need the following proposition, counting the number of $3$-stable sets of size $k$.
\begin{proposition}[\cite{Estrugo-Pastine-2021}, \cite{Talbot-03}]
\label{prop:3-stable-vertex-set}
Let $k \geq 2$ and $n \geq 3k$. Then
$$\left| \bstable{n}{k}{3} \right| = \frac{n}{k}\binom{n-2k-1}{k-1}.$$
\end{proposition}

\begin{proof}[Proof of Lemma~\ref{thm:reduction}]
By the coloring described in the introduction, 
it remains to prove the lower bound $\chi(\kgstable{n}{k}{3}) \ge n-3k +3$.     Assume for   contradiction   that $\chi(\kgstable{n}{k}{3}) \leq n-3k+2$, and let $c:V(\kgstable{n}{k}{3}) \to [n-3k+2]$ be a proper coloring of $\kgstable{n}{k}{3}$. 
We consider two cases.
    
    \textbf{Case 1:} There exists a color class $\cC$ of $c$ that is a star. That is, there exists $j \in [n]$  such that $j \in \bigcap_{X \in \cC} X$. Note that up to rotating the vertices, we may assume that $j = n$. Let $$\cF' = \binom{[n]}{k}_{3\text{-stable}}\setminus\cC.$$ Consider the $k$-vector $\vec{s} = (3, \dots, 3,2)$. Observe that $\bstable{n-1}{k}{\vec{s}} \subset \cF'. $ Indeed, if $A$ is an $\vec{s}$-stable $k$-set contained in $[n-1]$ then $A$ is a $3$-stable $k$-set in $[n],$ and since $n \notin A$, we have $A \not\in \cC$.

Therefore, $\kgstable{n-1}{k}{\vec{s}} \subseteq \text{KG}(\cF').$ By the negation assumption $\chi(\text{KG}(\cF')) \le n-3k+1$, and thus we have $$\chi(\kgstable{n-1}{k}{\vec{s}}) \le  n-3k+1 = (n-1)-3k+2.$$ However, by Theorem~\ref{cor:almost-3stable}, we have $\chi(\kgstable{n-1}{k}{\vec{s}}) = (n-1)-3k+3$, a contradiction.

\textbf{Case 2.} Every color class of $c$ is not a star. Then if  $C$ is a color class,  by the assumption of the lemma we have $$|C| < \frac{n}{k(n-3k+2)}\binom{n-2k-1}{k-1}.$$
 By our negation assumption, this implies that the number of $3$-stable $k$-sets in $[n]$ is 
 \begin{align*}
 \left| \bstable{n}{k}{3} \right|< (n-3k+2)\cdot\frac{n}{k(n-3k+2)} \cdot \binom{n-2k-1}{k-1}=\frac{n}{k}\binom{n-2k-1}{k-1}, 
\end{align*}
contradicting Proposition \ref{prop:3-stable-vertex-set}.
\end{proof}

\section{Proof of Theorems \ref{thm:main:k>=4} and \ref{thm:k=3} }
\label{sec:k=3}

\begin{proof}[Proof of Theorem~\ref{thm:k=3}]
    Suppose first that $n \ge 15$.  By Lemma  \ref{thm:reduction}, it is enough to show that every non-star intersecting family $\cF \subset \bstable{n}{3}{3}$ satisfies
    $$|\mathcal{F}| < \frac{n}{3(n-7)}\binom{n-7}{2}.$$ 
However, by Theorem~\ref{thm:stable-HM} we have  $$|\cF| \le 3n-28< \frac{n}{3(n-7)}\binom{n-7}{2},$$ where the second inequality is true for all $n\ge 15$. 

We were able to check the remaining cases, that is, $9 \leq n \leq 14$, by converting the problem into a satisfiability problem and solving by computer, confirming that for all $n \geq 9$,
    $\chi(\kgstable{n}{3}{3}) = n - 6.$ See the Appendix 2 for more details. 
\end{proof}

Before proving Theorem~\ref{thm:main:k>=4}, we need one more lemma. The proof is in Appendix 1. 

\begin{restatable}{lemma}{threeone}
    \label{lem:three>one}
    Let $k \geq 4$ and $n \geq k^3+3k^2$. Then, we have
    $$\frac{n}{k(n-3k+2)}\binom{n-2k-1}{k-1} > \binom{n-2k-2}{k-2}+ \binom{n-2}{k-1} - \binom{n-k-1}{k-1}+1.$$
\end{restatable}
We are now ready to prove Theorem~\ref{thm:main:k>=4}.
\begin{proof}[Proof of Theorem~\ref{thm:main:k>=4}]
    Let $k \geq 4$ and  $n \geq k^3 + 3k^2$. Set
    $$c(n,k,3):= \binom{n-2k-2}{k-2}+ \binom{n-2}{k-1} - \binom{n-k-1}{k-1}+1.$$
By Theorem~\ref{thm:tau-atleast-two}, any non-star intersecting family $\cF \subseteq \bstable{n}{k}{3}$ has size at most $c(n,k,3) < \frac{n}{k(n-3k+2)}\binom{n-2k-1}{k-1} $, where the inequality follows from 
Lemma~\ref{lem:three>one}. 

Now, applying Lemma~\ref{thm:reduction}, we obtain 
    $\chi(\kgstable{n}{k}{3}) = n - 3k + 3.$ 
\end{proof}



\section{A topological approach}
\label{sec:top}

The proofs in the previous sections are primarily combinatorial, relying on extremal set theory. Nevertheless, topology still plays an essential role through Theorem \ref{cor:almost-3stable}, whose proof ultimately relies on embeddings of neighborhood complexes. This naturally raises the question of whether Meunier's conjecture admits a more direct topological proof, in the spirit of Lov\'asz's proof of Kneser's conjecture and Schrijver's proof for 2-stable Kneser graphs. Although we were unable to develop such a proof in full generality, the discussion below offers a different topological perspective on stable Kneser graphs. In particular, it provides a new proof of the conjecture in the case where $k=2$ and $s-1|n$. While the method does not currently appear to extend to arbitrary values of $n$ and $k$, we believe that the topological viewpoint is of independent interest and may prove useful in future investigations.

Given $\cF \subset 2^{[n]}$, define a simplicial complex $\sfK(\cF):=\{A \subset [n]\ |\ B \not\subseteq A \ \forall B \in \cF\}.$ We say that a simplicial complex $\sfK$ {\em embeds} into $\R^d$ if there is a continuous map $f:\|\sfK(\cF)\| \to \R^d$ such that $f(\sigma_A) \cap f(\sigma_B) = \emptyset$ for every disjoint faces $\sigma_A,\sigma_B$ of $\sfK(\cF).$ The function $f$ is called an {\em embedding}.

We will use the following consequence of Sarkaria's embedding theorem.

\begin{theorem}[\cite{Matouvsek-03,Sarkaria-90, Sarkaria-91}]
\label{thm:sarkaria's embedding}
Let $\cF \subseteq 2^{[n]}$. If $\sfK(\cF)$ embeds into $\R^d$, then $\chi(\KG(\cF)) \geq n-d-1.$     
\end{theorem}
Therefore, to show that $\chi(\kgstable{n}{2}{s}) \geq n-s$, it suffices to find an embedding  of $\sfK(\bstable{n}{2}{s})$ in $\R^{s-1}$. 
We first give a description of the simplicial complex $\sfK(\bstable{n}{2}{s}).$
\begin{lemma}
\label{lem: description-of-K}
    Let $\sfK = \sfK(\bstable{n}{2}{s})$ with $n \geq 3s-2$ and $s \geq 3$. If $A \in \sfK$ is a facet, then $A$ is an interval on size $s$ in the cycle on vertex set $[n]$, that is, $A$ has the form $\{i,i+1,\dots,i+s-1\},$ where the addition is modulo $n$. In particular, every facets of $\sfK$ has size $s$.
\end{lemma}
\begin{proof}
    Let $B \in \sfK$ be a face. We show that there is an interval $\{i,i+1,\dots,i+s-1\}$ of size $s$ containing $B.$ Let $b \in B$. By rotating, we may assume $b =s .$ Since $B$ does not contain any $s$-stable set of size $2$, $B \subset [1,2s-1].$ Let $b_1 \in B$ be the smallest element and $b_2 \in B$ be the largest element. Since $d(b_1,b_2) \leq s-1,$ we either have $b_2 - b_1 \leq s-1$ or $n-(b_2-b_1) \leq s-1.$ If the latter case holds, then $$n \leq s-1 + b_2 - b_1 \leq s-1+2s-2 = 3s-3,$$ a contradiction. Thus, $b_2 - b_1 \leq s-1.$ Then $B$ is contained in the $s$-interval $[b_1, b_1+s-1].$
    
    Since an $s$-interval does not contain any $s$-stable set of size $2$ and adding any element to an $s$-interval creates an $s$-stable set of size $2$, facets of $\sfK$ are exactly $s$-intervals.    
\end{proof}

To embed $\sfK:=\sfK(\bstable{n}{2}{s})$ into $\R^{s-1},$ we first define an auxiliary simplicial complex $\tilde{\sfK}$ whose vertex set is $[n+s-1]$ and facets are $\{i,i+1,\dots,i+s-1\}$ for $1 \leq i \leq n$. For every $i\in [n+s-1]$ let $u_i$ be the vertex of $\|\tilde\sfK\|$ corresponding to the $i$-th vertex of $\tilde\sfK$. We will first show that $\|\tilde{\sfK}\|$ embeds into $\R^{s-1}$. 
Later we will identify each vertex $\{n+j\}$ with $\{j\}$ for $1 \leq j \leq s-1$ in $\tilde{\sfK}$  and obtain an embedding of $\|\sfK\|$ into $\R^{s-1}.$ See Figure \ref{fig:topological-approach}.

     \begin{center}
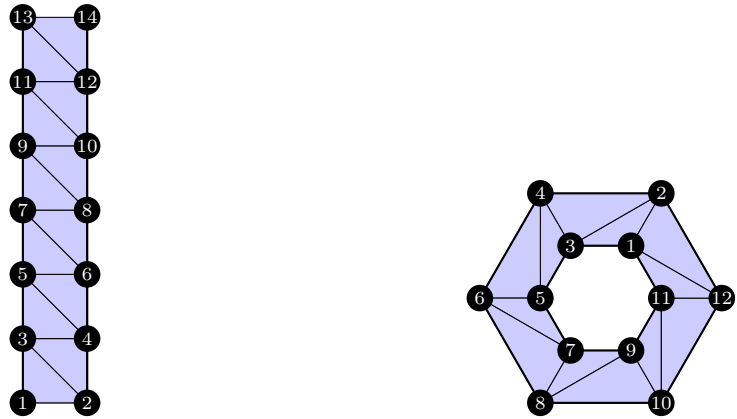
\begin{figure}[ht]
    \begin{subfigure}[b]{0.48 \textwidth}
        \begin{center}
\begin{tikzpicture}[scale=0.85]
\def\labelsize{3pt}

\foreach \k in {0,...,6}{
    \pgfmathtruncatemacro{\odd}{2*\k+1}
    \pgfmathtruncatemacro{\even}{2*\k+2}
    \coordinate (V\odd) at (0,\k);
    \coordinate (V\even) at (1,\k);
}

\draw[->] (0,0) -- (1,0) node[right] {};
\draw[->] (0,0) -- (0,6) node[above] {};

\draw[thick] (1,0) -- (1,6) node[above] {};

\foreach \i in {1,...,12}{
    \pgfmathtruncatemacro{\j}{\i+1}
    \pgfmathtruncatemacro{\k}{\i+2}
    \fill[blue!20] (V\i) -- (V\j) -- (V\k) -- cycle;
    \draw (V\i) -- (V\j) -- (V\k) -- cycle;
}

\draw[thick] (V1)--(V3)--(V5)--(V7)--(V9)--(V11)--(V13);
\draw[thick] (V2)--(V4)--(V6)--(V8)--(V10)--(V12)--(V14);

\foreach \i in {1,...,14}{
    \node[circle, fill=black, text=white,
          minimum size=10pt, font=\scriptsize, inner sep=0pt] at (V\i) {\i};
}

\end{tikzpicture}
\end{center}
    \end{subfigure}
    \begin{subfigure}[b]{0.48 \textwidth}
                \begin{center}
\begin{tikzpicture}[scale=0.8]
\def\rone{1}
\def\rtwo{2}
\def\labelsize{3pt} 

\foreach \i in {1,...,6}{
    \pgfmathsetmacro{\angle}{60+60*(\i-1)}
    \coordinate (M1-\i) at (\angle:\rone);
    \coordinate (M2-\i) at (\angle:\rtwo);
}

\foreach \i in {1,...,6}{
    \pgfmathtruncatemacro{\odd}{2*\i-1}
    \pgfmathtruncatemacro{\even}{2*\i}
    \coordinate (V\odd) at (M1-\i);
    \coordinate (V\even) at (M2-\i);
}

\foreach \i in {1,...,12}{
    \pgfmathtruncatemacro{\j}{mod(\i,12)+1}
    \pgfmathtruncatemacro{\k}{mod(\i+1,12)+1}
    \fill[blue!20] (V\i) -- (V\j) -- (V\k) -- cycle;
    \draw (V\i) -- (V\j) -- (V\k) -- cycle;
}

\draw[thick] (M1-1)--(M1-2)--(M1-3)--(M1-4)--(M1-5)--(M1-6)--cycle;
\draw[thick] (M2-1)--(M2-2)--(M2-3)--(M2-4)--(M2-5)--(M2-6)--cycle;

\foreach \i in {1,...,12}{
    \node[circle, fill=black, text=white, minimum size=10pt, font=\scriptsize, inner sep=0pt] at (V\i) {\i};
}

\end{tikzpicture}
\end{center}
    \end{subfigure}
    \caption{The left figure illustrates $\|\tilde \sfK\|$ when $n = 12$ and $s=3$, embedded into $\R^2$. After identifying $1\sim 13$ and $2 \sim 14$, we get the complex $\|\sfK\|$ embedded into $\R^2$ in the right figure.}
    \label{fig:topological-approach}
\end{figure}
\end{center}
Write $\overline{j}=j\mod (s-1)$. Define a map $f_1: V(\|\tilde{\sfK}\|) \to \R^{s-1}$ by $f_1(u_i) = e_{\overline{i-1}}+ \lfloor \frac{i-1}{s-1}\rfloor e_{s-1},$ where $e_0$ is the origin of $\R^{s-1}$ and $e_1,\dots,e_{s-1}$ are the standard basis vectors in $\R^{s-1}$. 
Note that $f_1$ maps $V(\|\tilde{\sfK}\|)$ into the $s-1$ vertical lines passing through $e_0,e_1, \dots, e_{s-2}$.
Let $f:\|\tilde{\sfK}\| \to \R^{s-1}$ be the linear extension of $f_1.$ 
Clearly, $f$ is continuous.

\begin{lemma}
\label{lem: tildeK is homeomorphic to}
    Fix $n \geq 3s-2$, $s \geq 3$, and let  $\tilde{\sfK}$ and $f$ be as above. Then $f$ is an embedding of $\tilde \sfK$ into $\R^{s-1}$. Moreover, when $s-1| n$, $f$ is a homeomorphism $\|\tilde \sfK\| \cong \conv\{e_{0},\dots,e_{s-2}\} \times [0, \frac{n}{s-1}].$
\end{lemma}
\begin{proof}
     Let $\sigma = \{a_1,\dots,a_l\}$ and $\tau=\{b_1,\dots,b_k\}$ be two disjoint faces of $\tilde\sfK$, where $a_i,b_j \in [n+s-1]$ for all $i,j$, and $a_1< \cdots < a_l$ and $b_1 < \cdots < b_k$. Without loss of generality, assume $a_1 < b_1.$ Write $a_1-1=(s-1)q+r$ where $r = \overline{a_1-1}$ and $q = \lfloor \frac{a_1-1}{s-1}\rfloor.$ Then by our definition, $$f(a_1) = e_r + qe_{s-1}.$$ For $0\leq i \leq s-2$, define the vertical lines $L_i = \{e_i + c e_{s-1}: c \in \R \}.$ For $j \geq 1$, let $H_j$ be the hyperplane $\{x \in \R^{s-1}: x_{s-1}=j\}$, and let $P_{i,j}$  be  the intersection point of the line $L_i$ with the hyperplane $H_j$. Then by definition, $f(a_1) = P_{r,q}$. Note that the image of the vertices of  a facet of $\tilde\sfK$ lie in exactly two consecutive hyperplanes $H_{j} \cup H_{j+1}$. In particular, $f(a_1)=P_{r,q} \in H_{q}$ and 
     \begin{equation}
     \label{eq:sigma-belongs}
     f(V(\sigma)) \subset \bigcup_{j=r}^{s-2}\{P_{j,q}\} \sqcup \bigcup_{j=0}^r \{P_{j,q+1}\} \subset H_{q} \cup H_{q+1}.    
     \end{equation}
     
     Note that since $a_1<b_1$, we have $f(b) \in H_k$ with $k\ge q$.

\noindent{\bf Claim 1.}
     Let $a\in \sigma$ and $b \in \tau$, and suppose that  $f(a)\in L_i \cap H_m$ and $f(b)\in L_i \cap H_k$. We claim that $k> m$.

    Indeed, assume for contradiction $k\le m$. Since $\sigma$ and $\tau$ are disjoint, $f(a)\neq f(b)$, and thus $k<m$. By (\ref{eq:sigma-belongs}) we may consider two cases.  If $f(a) \in  \bigcup_{j=r}^{s-2}\{P_{j,q}\}$, then $m=q$, and thus $f(b) \in H_k$ for  $k \leq q-1$, implying $b < a_1$, contradicting  $a_1 < b_1$. Otherwise, $f(a) \in \bigcup_{j=0}^r \{P_{j,q+1}\} \subseteq H_{q+1}$  so  $f(b)$ must be in $H_q \cap (L_0\cup \dots \cup L_{r-1})$ . But then again, $b < a_1,$ a contradiction. See Figure \ref{fig:Hyperplanes-lines}.

    \begin{figure}[ht]
        \centering
        \begin{tikzpicture}[
    scale=1.1,
    line cap=round,
    line join=round
]

\coordinate (X1) at (1.3,0);
\coordinate (X2) at (2.3,0);
\coordinate (X3) at (3.3,0);
\coordinate (X4) at (4.3,0);
\coordinate (X5) at (5.3,0);
\coordinate (X6) at (6.3,0);
\coordinate (X7) at (7.3,0);

\def\lowerfront{0}
\def\lowerback{0.55}
\def\lowerpoint{0.275}

\def\upperfront{2.30}
\def\upperback{2.85}
\def\upperpoint{2.575}

\filldraw[
    fill=blue!15,
    fill opacity=0.12,
    draw=black,
    draw opacity=1,
    thick
]
    (0.5,\lowerfront)
    -- (8.1,\lowerfront)
    -- (9.0,\lowerback)
    -- (1.4,\lowerback)
    -- cycle;

\filldraw[
    fill=blue!15,
    fill opacity=0.12,
    draw=black,
    draw opacity=1,
    thick
]
    (0.5,\upperfront)
    -- (8.1,\upperfront)
    -- (9.0,\upperback)
    -- (1.4,\upperback)
    -- cycle;


\draw[thick] (1.3,3.60) -- (1.3,\upperpoint);
\draw[thick] (1.3,\upperfront) -- (1.3,\lowerpoint);

\draw[thick] (2.3,3.60) -- (2.3,\upperpoint);
\draw[thick] (2.3,\upperfront) -- (2.3,\lowerpoint);

\draw[thick] (3.3,3.60) -- (3.3,\upperpoint);
\draw[thick] (3.3,\upperfront) -- (3.3,\lowerpoint);

\draw[thick] (4.3,3.60) -- (4.3,\upperpoint);
\draw[thick] (4.3,\upperfront) -- (4.3,\lowerpoint);

\draw[thick] (5.3,3.60) -- (5.3,\upperpoint);
\draw[thick] (5.3,\upperfront) -- (5.3,\lowerpoint);

\draw[thick] (6.3,3.60) -- (6.3,\upperpoint);
\draw[thick] (6.3,\upperfront) -- (6.3,\lowerpoint);

\draw[thick] (7.3,3.60) -- (7.3,\upperpoint);
\draw[thick] (7.3,\upperfront) -- (7.3,\lowerpoint);


\foreach \x in {4.3,5.3,7.3}{
    \filldraw[
        fill=red,
        draw=red!60!black,
        thick
    ]
    (\x,\lowerpoint) circle (3.2pt);
}

\foreach \x in {6.3}{
    \node[
        diamond,
        fill=green!65!black,
        draw=green!35!black,
        thick,
        minimum size=8pt,
        inner sep=0pt
    ] at (\x,\lowerpoint) {};
}

\foreach \x in {1.3,3.3,4.3,6.3}{
    \node[
        diamond,
        fill=green!65!black,
        draw=green!35!black,
        thick,
        minimum size=8pt,
        inner sep=0pt
    ] at (\x,\upperpoint) {};
}
\foreach \x in {2.3,3.3}{
    \filldraw[
        fill=red,
        draw=red!60!black,
        thick
    ]
    (\x,\upperpoint) circle (3.2pt);
}

\node[above] at (4.3,3.60) {$L_r$};
\node[above] at (1.3,3.60) {$L_0$};
\node[above] at (7.3,3.60) {$L_{s-2}$};
\node[right] at (8.6,2.47) {$H_{q+1}$};
\node[right] at (8.6,0.17) {$H_q$};
\node[right] at (8.6,2.47) {$H_{q+1}$};
\node[below] at (4.3,0) {$P_{r,q}$};

\end{tikzpicture}
        \caption{The vertices of $\sigma$ and $\tau$ are mapped under $f$ to red circles and green diamonds, respectively.}
        \label{fig:Hyperplanes-lines}
    \end{figure}
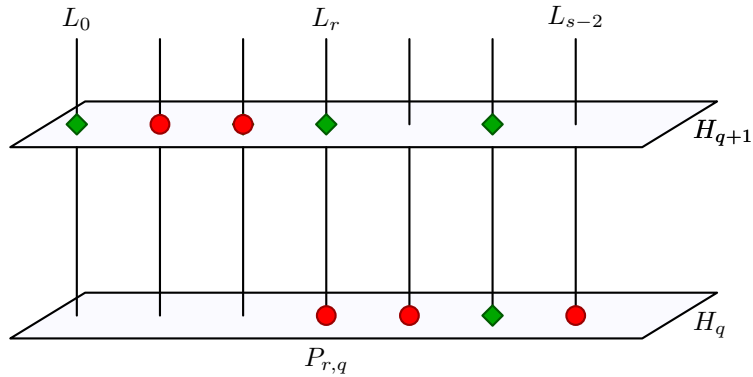
    \medskip
    
   \noindent{\bf Claim 2.} There exists a hyperplane separating   $f(\sigma)$ from $f(\tau)$.

   Indeed, note that by Claim 1,  on each line $L_i$, every point from $f(V(\sigma))$ is below every point from $f(V(\tau))$.  
   For each line $L_i$ with $0 \leq i \leq s-2$, we define a point $Q_i$ to be an arbitrary point that is larger than every point of $L_i\cap f(V(\sigma))$ and smaller than every point of $L_i\cap f(V(\tau))$. Let $H \subset \R^{s-1}$ be a hyperplane spanned by $Q_1,\dots,Q_{s-1}.$ Then $H$ separates $f(\sigma)$ from $f(\tau).$ 

 



    This completes that proof that $f$ is an embedding of $\tilde\sfK$ into $\R^{s-1}$. Since the image of $f$
    when $s-1|n$ is $\conv\{e_{0},\dots,e_{s-2}\} \times [0, \frac{n}{s-1}]$, we have in that case that $f$ 
 is a  homeomorphism $\|\tilde \sfK\| \cong \conv\{e_{0},\dots,e_{s-2}\} \times [0, \frac{n}{s-1}]$, as needed.
\end{proof}

By construction, $\sfK$ is obtained from $\tilde\sfK$  by identifying $\{j\}$ with $\{n+j\}$, for $1 \leq j \leq s-1$. 
By the previous lemma we have  $$\|\tilde\sfK\|\cong  \conv\{e_{0},\dots,e_{s-2}\} \times [0, \frac{n}{s-1}] \cong \conv(e_0,\dots,e_{s-2}) \times [0,1].$$ 
Note that in this homeomorphism, for every $1 \leq j \leq s-1$, $u_j \mapsto  (e_{j-1},0) \mapsto  (e_{j-1},0)$ and  $ u_{n+j} \mapsto  (e_{j-1},\frac{n}{s-1}) \mapsto (e_{j-1},1)$. Therefore, to show that $\|\sfK\|$ can be embedded into $\R^{s-1}$, it suffices to show the quotient space $\|\tilde\sfK\|/{\sim}$ can be embedded into $\R^{s-1},$ where $(x,0) \sim (x,1)$ for $x \in \conv(e_0,\dots,e_{s-2}).$

Now observe that 
$$\|\tilde\sfK\|/{\sim} \cong  \conv(e_0,\dots,e_{s-2}) \times \mathbb{S}^1 \cong \Delta^{s-2}\times \mathbb{S}^1 \cong [0,1]^{s-2} \times \mathbb{S}^1,$$
where $\Delta^{s-2}$ is the $(s-2)$-simplex, $[0,1]^{s-2}$ is the $(s-2)$-cube, and $\mathbb{S}^1$ is $1$-sphere. Thus it suffices to show the following.
\begin{lemma}
\label{lem: embed cube x sphere}
For $s \geq 3$, $[0,1]^{s-2} \times \mathbb{S}^1$ can be embedded into $\R^{s-1}.$
\end{lemma}
\begin{proof}
    We proceed by induction. When $s=3,$ $[0,1] \times \mathbb{S}^1$ is homeomorphic to an annulus, which can be embedded into $\R^2.$ Assume $s \geq 4$ and the statement holds for smaller $s.$ Since $[0,1]^{s-3} \times \bS^1$ can be embedded into $\R^{s-2}$, $[0,1]^{s-2} \times \bS^1$ can be embedded into $[0,1] \times \R^{s-2} \hookrightarrow \R^{s-1},$ as desired.
\end{proof}

By Theorem \ref{thm:sarkaria's embedding}, we obtain a topological proof to the following statement.

\begin{theorem}
    Let $n \geq 3s-2$, $s \geq 3$, and $s-1 | n.$ Then $\chi(\kgstable{n}{2}{s}) = n-s.$
\end{theorem}

\bibliography{ref.bib}
\bibliographystyle{plain}

\appendix
\section*{Appendix 1: Proofs}
\begin{proof}[Proof of Lemma \ref{lem:tau-three-less-than}]
    By Pascal's inequality, 
    $$\binom{n-2}{k-1} - \binom{n-k-1}{k-1} = \sum_{j = 1}^{k-1} \binom{n-2-j}{k-2}.$$
For $1 \leq j \leq k - 1$, let 
    $R_j = \frac{\binom{n-2-j}{k-2}}{\binom{n-3}{k-3}}.$ Then $$R_1 = \frac{n-k}{k-2}.$$ For $j \geq 2,$ we have
    \begin{align*}
        R_j &= \frac{\binom{n-2-j}{k-2}}{\binom{n-3}{k-3}} 
        = \left( \frac{1}{k-2} \right) \left( \frac{(n-k)!}{(n-k-j)!}\right) \left( \frac{(n-2-j)!}{(n-3)!}\right) \\
        &=  \frac{1}{k-2} \cdot (n-k)(n-k-1) \cdots (n-k-j+1) \cdot \frac{1}{(n-3)(n-4)\cdots(n-j-1)} \\
        &=  \frac{n-k}{k-2} \cdot \prod_{i = 1}^{j-1} \frac{n-k-i}{n-2-i} =  \frac{n-k}{k-2} \cdot \prod_{i = 1}^{j-1} \left( 1 - \frac{k-2}{n-2-i} \right).
    \end{align*}

    By applying the Weierstrass product inequality, and since $i+2\le j+1\le k$, we have
    \begin{align*}
       \frac{n-k}{k-2}  \cdot \prod_{i = 1}^{j-1}& \left( 1 - \frac{k-2}{n-2-i} \right)
        \geq\frac{n-k}{k-2}  \left(1 - \sum_{i = 1}^{j-1} \frac{k-2}{n-2-i} \right) \\
        &\geq\frac{n-k}{k-2}  \left(1 - \sum_{i = 1}^{j-1} \frac{k-2}{n-k} \right) 
        =\frac{n-k}{k-2}  \left(1 - \frac{(j-1)(k-2)}{n-k} \right) =\frac{n-k}{k-2}  - (j-1).
    \end{align*}
  Thus we obtain
    $$\sum_{j = 1}^{k-1} R_j \geq  \frac{(k-1)(n-k)}{k-2}  - \sum_{j = 2}^{k-1} (j - 1) = \frac{(k-1)(n-k)}{k-2} - \frac{(k-2)(k-1)}{2}.$$

To prove the desired inequality, we need to show that $\sum_{j = 1}^{k-1} R_j \geq k^3$.
Set $T(n,k) := \frac{(k-1)(n-k)}{k-2} - \frac{(k-2)(k-1)}{2}$.  It suffices to show $T(n,k) \ge k^3$. Since $T(n,k)$ is increasing in $n$, it is enough to   show that $T(k^3,k) \geq k^3$. Indeed, by applying long division we get
    \begin{align*}
        T(k^3,k) &= \frac{(k-1)(k^3-k)}{k-2} - \frac{(k-2)(k-1)}{2} \\
        &= \frac{k(k-1)^2(k+1)}{k-2} - \frac{(k-2)(k-1)}{2} \\
        &= k^3 + k^2 + k + 3+\frac{6}{k-2}  - 0.5k^2 + 1.5k - 1 \geq k^3.
    \end{align*} 
\end{proof}

\begin{proof}[Proof of Lemma \ref{lem:three>one}]
We begin by showing that
$$\frac{n}{k(n-3k+2)}\binom{n-2k-1}{k-1} > \frac{n^{k-1}}{k!}\left(1- \frac{3}{k+3} \right).$$
First, we obtain a sequence of lower bounds on $\frac{n}{k(n-3k+2)}\binom{n-2k-1}{k-1}$. We have
\begin{align*}
    \frac{n}{k(n-3k+2)}\binom{n-2k-1}{k-1} &= \frac{n}{(n-2k)(n-3k+2)}\binom{n-2k}{k} \\
    & \geq \frac{n}{(n-2k)(n-3k+2)}\frac{(n-3k+2)^{k-1}(n-3k+1)}{k!} \\
    &= \frac{n}{n-2k}\frac{(n-3k+2)^{k-2}(n-3k+1)}{k!} \\
    &> \frac{(n-3k+2)^{k-2}(n-3k+1)}{k!} \\
    &> \frac{(n-3k+1)^{k-1}}{k!} = \frac{n^{k-1}}{k!}\left(1-\frac{3k-1}{n}\right)^{k-1}.
\end{align*}
Note that $(1+x)^r \geq 1 + rx$ whenever $x \geq -1$ and $r \geq 0$. Therefore,
\begin{align*}
    \frac{n^{k-1}}{k!}(1-\frac{3k-1}{n})^{k-1} &\geq \frac{n^{k-1}}{k!}\left(1-\frac{(3k-1)(k-1)}{n} \right) \\
    & > \frac{n^{k-1}}{k!}\left(1-\frac{3k^2}{n} \right) \\
    &\geq \frac{n^{k-1}}{k!}\left(1-\frac{3k^2}{k^3+3k^2} \right) \\
    &= \frac{n^{k-1}}{k!}\left(1-\frac{3}{k+3} \right).
\end{align*}

In the last inequality, we use the fact that $n \geq k^3 + 3k^2$. 

Next, we show that $\frac{n^{k-1}}{k!}\left(1-\frac{3}{k+3} \right) \geq \binom{n-2k-2}{k-2}+ \binom{n-2}{k-1} - \binom{n-k-1}{k-1}+1$. In doing so, we will obtain the desired inequality.

First, observe that by iteratively applying Pascal's identity, we get:
$$\binom{n-2}{k-1}-\binom{n-k-1}{k-1} =\sum_{l=1}^{k-1}\binom{n-2-l}{k-2} \leq (k-1)\frac{n^{k-2}}{(k-2)!}.$$ 

Therefore, we have:
\begin{align*}
    \binom{n-2k-2}{k-2} &+ \binom{n-2}{k-1} - \binom{n-k-1}{k-1}+1 \leq \frac{n^{k-2}}{(k-2)!} + (k-1)\frac{n^{k-2}}{(k-2)!}+1 \\
    &\leq (k+1)\frac{n^{k-2}}{(k-2)!} = \frac{n^{k-1}}{k!}\frac{(k+1)k(k-1)}{n} \leq \frac{n^{k-1}}{k!}\frac{k^3}{n} \\
    &\leq \frac{n^{k-1}}{k!}\frac{k^3}{k^3+3k^2} = \frac{n^{k-1}}{k!}\left(1- \frac{3}{k+3} \right).
\end{align*}

Again, in the last inequality, we used the assumption that $n \geq k^3 + 3k^2$. Putting everything together,  the proof is complete.
\end{proof}

\section*{Appendix 2: Code }
For the remaining values $9\le n \le 14$, our extremal argument does not provide sufficiently strong bounds. We therefore verified these finite cases by an exhaustive satisfiability computation. We formulate the existence of a proper coloring with $n-7$ colors as a SAT instance and solve it using the CaDiCaL solver. In each case the instance is unsatisfiable, establishing that $\chi(\kgstable{n}{3}{3}) \ge n-6$. For completeness and reproducibility, the implementation used for these computations is included here.\\

\begin{lstlisting}
import itertools as it
import math
from pysat.solvers import Cadical195

#generate vertex set and edge set of KG(n,k)_{s-stable}
def gen_stable_sets(n, k, s):
    G = list()
    vtx_set = []
    for x in it.combinations(range(1,n+1),k):
        y = sorted(list(x))
        isSparse = True
        for i in range(1, k):
            if min(y[i] - y[i - 1], n + y[i - 1] - y[i]) < s:
                isSparse = False
                break
                
        if isSparse and min(y[k - 1] - y[0], n + y[0] - y[k - 1]) >= s:
                vtx_set.append(x)

    edge_set = []
    for elt in it.combinations(vtx_set,2):
        if set(elt[0]).isdisjoint(set(elt[1])):
            edge_set.append(elt)

    return vtx_set, edge_set
    
#build and solve satisfiability program, checking if there is a proper 
#of KG(n,r)_{s-stable} coloring using c colors
def build_sat_instances(n, k, s, c):
    
    g = Cadical195()
    (V, E) = gen_stable_sets(n, k, s)
    c_wise_vertices = [[c*i+j for j in range(1,c+1)] for i in range(len(V))]
                
    #Every vertex must be assigned some color
    for i in range(len(V)):
        g.add_clause(c_wise_vertices[i])

    #No vertex has two colors
    for i in range(len(V)):
        for (a,b) in it.combinations(c_wise_vertices[i], 2):
            g.add_clause([-a, -b])
        
    #No adjacent vertices have the same color
    for (a,b) in E:
        a_idx = V.index(a)
        b_idx = V.index(b)

        for j in range(c):
            g.add_clause([-1*c_wise_vertices[a_idx][j], -1*c_wise_vertices[b_idx][j]])

    
    #solve SAT problem, assuming WLOG the first edge is colored with the first two colors
    x = g.solve(assumptions = [c_wise_vertices[V.index(E[0][0])][0], c_wise_vertices[V.index(E[0][1])][1]])
    print(x)


# Uncomment to run the code

# build_sat_instances(9, 3, 3, 9 - 3*(3-1)-1)
#  output: False
# build_sat_instances(10, 3, 3, 10 - 3*(3-1)-1)
#  output: False
# build_sat_instances(11, 3, 3, 11 - 3*(3-1)-1)
#  output: False
# build_sat_instances(12, 3, 3, 12 - 3*(3-1)-1)
#  output: False
# build_sat_instances(13, 3, 3, 13 - 3*(3-1)-1)
#  output: False
# build_sat_instances(14, 3, 3, 14 - 3*(3-1)-1)
#  output: False
\end{lstlisting}

\end{document}